\documentclass{amsart}

 \newtheorem{thm}{Theorem}[section]
 \newtheorem{cor}[thm]{Corollary}
 \newtheorem{lem}[thm]{Lemma}
 \newtheorem{prop}[thm]{Proposition}
 \theoremstyle{definition}
 
 \theoremstyle{remark}
 \newtheorem{rem}[thm]{Remark}
 
 \numberwithin{equation}{section}

\newcommand{\A}{\mathcal{A}}
\newcommand{\B}{\mathcal{B}}

\newcommand{\M}{\mathcal{M}}

\newcommand{\X}{\mathcal{X}}
\newcommand{\N}{\mathcal{N}}

\newcommand{\PH}{\phi:\mathcal{A}\times \mathcal{A}\rightarrow \mathcal{X}}
\begin{document}

\title[On derivations and Jordan derivations...]
 {On derivations and Jordan derivations through zero products}

\author{ Hoger Ghahramani}

\thanks{{\scriptsize
\hskip -0.4 true cm \emph{MSC(2010)}: 15A86; 47A07; 47B47; 47B49.
\newline \emph{Keywords}: Bilinear maps, Derivation, Jordan derivation, Zero (Jordan) product determined algebra. \\}}

\address{Department of
Mathematics, University of Kurdistan, P. O. Box 416, Sanandaj,
Iran.}

\email{h.ghahramani@uok.ac.ir; hoger.ghahramani@yahoo.com}

\thanks{}

\thanks{}

\subjclass{}

\keywords{}

\date{}

\dedicatory{}

\commby{}


\begin{abstract}
Let $\A$ be a unital complex (Banach) algebra and $\M$ be a unital
(Banach) $\A$-bimodule. The main results describe (continuous)
derivations or Jordan derivations $D:\A\rightarrow \M$ through
the action on zero products, under certain conditions on $\A$ and
$\M$. The proof is based on the consideration of a (continuous)
bilinear map satisfying a related condition.
\end{abstract}

\maketitle
\section{Introduction}
Throughout this paper all algebras and vector spaces will be over
the complex field $\mathbb{C}$ and all algebras are associative
with unity, unless indicated otherwise. All modules are unital.
Let $\A$ be an algebra and $\M$ be an $\A$-bimodule. Recall that
a linear map $D:\A \rightarrow \M$ is said to be a \emph{Jordan
derivation} (or \emph{generalized Jordan derivation}) if
$D(a\circ b)=D(a)\bullet b+a \bullet D(b)$ (or $D(a\circ b
)=D(a)\bullet b+a\bullet D(b)-aD(1)b-bD(1)a$) for all $a,b\in
\mathcal{A}$.
\\
Here and subsequently, $'\circ'$ denotes the Jordan product
$a\circ b=ab+ba$ on $\A$ and $'\bullet'$ denotes the Jordan
product on $\M$:
\[ a\bullet m=m\bullet a=am+ma, \quad a\in \A, \, m\in \M.\]
$D$ is called a \emph{derivation} (or \emph{generalized
derivation}) if $D(ab)=D(a)b+aD(b)$ (or
$D(ab)=D(a)b+aD(b)-aD(1)b$) for all $a,b\in \mathcal{A}$. Clearly,
each (generalized) derivation is a (generalized) Jordan
derivation. The converse is, in general, not true.
\par
The question of characterizing derivations or Jordan derivations
on algebras through the action on zero products has attracted the
attention of many authors over the last few years. We refer the
reader to \cite{Ala2, Gha} for a full account of the topic and a
list of references.
\par
In this paper, we consider the subsequent conditions on a linear
map $D$ from an algebra $\mathcal{A}$ into an
$\mathcal{A}$-bimodule $\mathcal{M}$:
\begin{enumerate}
\item[(d1)] $ab=0 \Rightarrow aD(b)+D(a)b=0.$
\item[(d2)] $ab=ba=0 \Rightarrow aD(b)+D(a)b=0.$
\item[(d3)] $a\circ b=0 \Rightarrow a\bullet D(b)+D(a)\bullet b=0.$
\item[(d4)] $ab=ba=0 \Rightarrow a\bullet D(b)+D(a)\bullet b=0.$
\end{enumerate}
Our purpose is to investigate whether these conditions
characterizes derivations or Jordan derivations.
\par
The above questions and the question of characterizing linear
maps that preserve zero products, Jordan product, etc. on algebras
can be sometimes effectively solved by considering bilinear maps
that preserve certain zero product properties (for instance, see
\cite{Ala1, Ala2, Ala3, Bre2, Chu}). Motivated by these reasons
Bre$\check{\textrm{s}}$ar et al. \cite{Bre} introduced the
concept of zero product (Jordan product) determined algebras,
which can be used to study the linear maps preserving zero
product (Jordan product) and derivable (Jordan derivable) maps at
zero point.
\par
In this context one is usually involved with the following
condition on a bilinear map $\PH$, where $\X$ is an arbitrary
linear space:
\[ a,b\in \A, \quad ab=ba=0 \Rightarrow \phi(a,b)=0. \quad \quad \qquad (G)\]
A way to unify and generalize both of the concepts of zero
product determined and zero Jordan product determined consists in
considering bilinear maps satisfying $(G)$.
\par
The paper is organized as follows. In section 2 we introduce the
notation and terminology, and then a class of (Banach)
$\A$-bimodules satisfying a condition $\mathbb{M}$
($\mathbb{M}^{\prime}$). Also we give several classes of bimodules
which satisfy this condition. Section 3 is concerned with bilinear
maps. We will consider the condition $(G)$ for bilinear maps in
this section. Also we present some results concerning the notions
of zero (Jordan) product determined algebras. In section 4 we
study the linear maps satisfying (d1)--(d4) for modules with
property $\mathbb{M}$ ($\mathbb{M}^{\prime}$), by using the
results of section 3.
\section{Preliminaries}
In this section we introduce the notation and terminology, and
then a special class of (Banach) bimodules.
\par
Let $\A$ be an algebra, then $\Im(\A)$ denotes the set of all
linear combinations of idempotents in $\A$. Let $\M$ be an
$\A$-bimodule. We say that $\M$ satisfies $\mathbb{M}$, if there
is an ideal $\mathcal{J}$ in $\A$ such that $\mathcal{J}\subseteq
\Im(\A)$ and
\begin{equation}\label{e}
\{m\in \M \, | \, xmx=0 \, \, for \, all \,\, x\in
\mathcal{J}\}=\{0\}.
\end{equation}
If $\A$ is a Banach algebra, $\M$ is a Banach $\A$-bimodule and
there is an ideal $\mathcal{J}$ in $\A$ such that
$\mathcal{J}\subseteq \overline{\Im(\A)}$ and \eqref{e} holds,
then we say that $\M$ satisfies $\mathbb{M}^{\prime}$.
\par
Note that if (Banach) $\A$-bimodule $\M$ satisfies $\mathbb{M}$
($\mathbb{M}^{\prime}$), then we have
\[ \{ m\in \M \, | \, xm=mx=0  \, \, for \, all \,\, x\in
\mathcal{J}\}=\{0\} . \] Now we introduce the class of (Banach)
bimodules with the property $\mathbb{M}$ ($\mathbb{M}^{\prime}$).
\begin{prop}\label{idem}
Let $\A$ be an (Banach) algebra with $\A=\Im(\A)$
($\A=\overline{\Im(\A)}$). Then every (Banach) $\A$-bimodule $\M$
satisfies $\mathbb{M}$ ($\mathbb{M}^{\prime}$)
\end{prop}
\begin{proof}
Let $m\in \M$ and $ama=0$ for all $a\in \A$. Since $\A$ is
unital, it follows that $m=0$. Now if we consider $\A$ as an
ideal, then by hypothesis any (Banach) $\A$-bimodule $\M$
satisfies $\mathbb{M}$ ($\mathbb{M}^{\prime}$).
\end{proof}
\par
If $\A$ is a $W^{*}$-algebra, then the linear span of projections
is norm dense in $\A$, so $\A=\overline{\Im(\A)}$.
\par
Let $\mathcal{H}$ be a Hilbert space and $B(\mathcal{H})$ denotes
the algebra of all bounded linear operators on $\mathcal{H}$. Then
from \cite[Lemma 3.2]{Hou} and \cite[Theorem 1]{Pea}, we have
$B(\mathcal{H})=\Im(B(\mathcal{H}))$. Recall that a
$W^{*}$-algebra is called \emph{properly infinite} if it contains
no nonzero finite central projection. Since every element in a
properly infinite $W^{*}$-algebra $\A$ is a sum of at most five
idempotents \cite[Theorem 4]{Pea}, it follows that $\A=\Im(\A)$.
\par
Let $\A$ be an algebra. Recall that a non-zero ideal $\mathcal{I}$
of $\A$ is called \emph{essential} if it has non-zero intersection
with every non-zero ideal of $\A$. The \emph{socle} of $\A$,
$Soc(A)$, is the sum of all minimal left ideals of $\A$, or
minimal right ideals of $\A$, if they exists; otherwise it is
zero. From Remark 2 of \cite{Bur} we have the next proposition.
\begin{prop}\label{semi}
Let $\A$ be a semisimple Banach algebra with non-zero socle. If
$Soc(\A)$ is essential, then $\A$ as an $\A$-bimodule satisfies
$\mathbb{M}$.
\end{prop}
Let $\X$ be a Banach space. We denote by $\B(\X)$ the algebra of
all bounded linear operators on $\X$, and $ \mathcal{F}(\X)$
denotes the algebra of all finite rank operators in $\B(\X)$.
Recall that a subalgebra $\A$ of the algebra $\B(\X)$ is called
\emph{standard} if $\A$ contains the identity and the ideal $
\mathcal{F}(\X)$. If $\A$ is a standard operator algebra on a
Banach space $\X$, then $\A$ is primitive and
$Soc(\A)=\mathcal{F}(\X)$ is essential. Thus,
Proposition~\ref{semi} applied for standard operator algebras.
\par
A \emph{nest} $\N$ on a Banach space $\X$ is a chain of closed
(under norm topology) subspaces of $\X$ which is closed under the
formation of arbitrary intersection and closed linear span
(denoted by $\vee$), and which includes $\{0\}$ and $\X$. The
\emph{nest algebra} associated to the nest $\N$, denoted by
$Alg\N$, is the weak closed operator algebra of the form
\[ Alg\mathcal{N}=\{ T\in \mathcal{B}(\X)\,|\, T(N)\subseteq N
\,for\, all \,N\in \mathcal{N}\}. \] When $\N \neq \{\{0\},X\}$,
we say that $\N$ is non-trivial. It is clear that if $\N$ is
trivial, then $Alg\N =\B(\X)$. Denote $Alg_{\mathcal{F}}\N:= Alg\N
\cap \mathcal{F}(\X)$, the set of all finite rank operators in
$Alg\N$ and for $N\in \N$, let $N_{-} = \vee \{M\in \N\,|\,M
\subset N \}$.
\begin{prop}\label{nest}
Let $\N$ be a nest on a Banach space $\X$. If $N\in \N$ is
complemented in $\X$ whenever $N_{-} = N$, then $\B(\X)$ as a
$Alg\N$-bimodule satisfies $\mathbb{M}$.
\end{prop}
\begin{proof}
$Alg_{\mathcal{F}}\N$ is an ideal of $Alg\N$ and from \cite{Hou},
it is contained in the $\Im(Alg\N)$. Suppose that $T\in \B(\X)$
and $FTF=0$ for each $F\in Alg_{\mathcal{F}}\N$. So we have
$(F_{1}+F_{2})T(F_{1}+F_{2})=0$ and hence
$F_{1}TF_{2}+F_{2}TF_{1}=0$, for any $F_{1},F_{2}\in
Alg_{\mathcal{F}}\N$. By \cite{Spa} we have
$\overline{Alg_{\mathcal{F}}\N}^{SOT}=Alg\N$. Therefore there is
a net $(F_{\gamma})_{\gamma \in \Gamma}$ in $Alg_{\mathcal{F}}\N$
converges to the identity operator $I$ with respect to the strong
operator topology. So $FT F_{\gamma}+F_{\gamma}T F=0$ for each
$\gamma \in \Gamma$ and $F\in Alg_{\mathcal{F}}\N$. Thus
$FT+TF=0$ for all $F\in Alg_{\mathcal{F}}\N$ and hence
$F_{\gamma}T+TF_{\gamma}=0$ for all $\gamma\in \Gamma$. So $T=0$
and $\B(\X)$ satisfies $\mathbb{M}$.
\end{proof}
It is obvious that the nests on Hilbert spaces, finite nests and
the nests having order-type $\omega + 1$ or $1 + \omega^{*}$,
where $\omega$ is the order-type of the natural numbers, satisfy
the condition in Proposition~\ref{nest} automatically.

\section{Bilinear maps vanishing on zero products}
In this section we concern with bilinear maps on algebras. From
this point up to the last section $\A$ is an algebra.
\par
The algebra $\A$ is called \emph{zero product determined} if for
every linear space $\X$ and every bilinear map $\PH$, the
following holds. If $\phi(a,b)=0$ whenever $ab=0$, then there
exists a linear map $T: \A \rightarrow \X$ such that $\phi(a,b)
=T(ab)$ for all $a,b\in \A$. If the ordinary product is replaced
by the Jordan product, then it is said that $\A$ is \emph{zero
Jordan product determined}.
\par
We will show that any unital Banach algebra spanned by idempotents
is zero product determined and zero Jordan product determined.
\begin{thm}\label{zp}
Let $\X$ be a linear space and let $\PH$ be a bilinear map
satisfying
\[ a,b\in \A, \quad ab=0 \Rightarrow \phi(a,b)=0.\]
Then
\[ \phi(a,x)=\phi(ax,1) \quad and \quad \phi(x,a)=\phi(1,xa) \]
for all $a\in \A$ and $x\in \Im(\A)$. Indeed, if $\A=\Im(\A)$,
then $\A$ is zero product determined.
\end{thm}
\begin{proof}
Let $a\in \A$. For arbitrary idempotent $p\in \A$, let $q=1-p$. We
have
\[\phi(a,p)=\phi(ap,p)+\phi(aq,p)=\phi(ap,p),\]
since $(aq)p=0$. On the other hand we have
\[\phi(ap,1)=\phi(ap,p)+\phi(ap,q)=\phi(ap,p).\]
By comparing the two expressions for $\phi(ap,p)$, we arrive at
$\phi(a,p)=\phi(ap,1)$. Since every $x\in \Im(\A)$ is a linear
combination of idempotent elements in $\A$, we get
\[ \phi(a,x)=\phi(ax,1) \]
for all $a\in \A$ and $x\in \Im(\A)$. Similarly, we get
$\phi(x,a)=\phi(1,xa)$ for all $a\in \A$ and $x\in \Im(\A)$.
\par
Now suppose that $\A=\Im(\A)$. Let $\X$ be a linear space, and let
$\PH$ be a bilinear map such that for all $a,b \in \A$, $ab= 0$
implies $\phi(a,b)=0$. From above identity we have
\[ \phi(a,b)=\phi(ab,1) \]
for all $a,b\in \A$, since $\A=\Im(\A)$. If we define the linear
map $T:\A \rightarrow \X$ by $T(a)=\phi(a,1)$, then $T$ satisfies
all the requirements in the definition of zero product determined
algebras. Thus $\A$ is a zero product determined algebra.
\end{proof}
\begin{prop}\label{bzp}
Let $\A$ be a Banach algebra, let $\X$ be a Banach space and let
$\PH$ be a continuous bilinear map satisfying
\[ a,b\in \A, \quad ab=0 \Rightarrow \phi(a,b)=0.\]
Then
\[ \phi(a,x)=\phi(ax,1) \quad and \quad \phi(x,a)=\phi(1,xa) \]
for all $a\in \A$ and $x\in \overline{\Im(\A)}$. If
$\A=\overline{\Im(\A)}$, then there exists a continuous linear
map $T: \A \rightarrow \X$ such that $\phi(a,b) =T(ab)$ for all
$a,b\in \A$.
\end{prop}
\begin{proof}
A similar proof as that of Theorem~\ref{zp} and the fact that
$\phi$ is continuous, shows that $\phi(a,x)=\phi(ax,1)$ for all
$a\in \A$ and $x\in \overline{\Im(\A)}$. If
$\A=\overline{\Im(\A)}$, we find
\[ \phi(a,b)=\phi(ab,1) \]
for all $a,b\in \A$. Now we define the linear mapping $T:\A
\rightarrow \X$ by $T(a)=\phi(a,1)$. So we have $\phi(a,b)=T(ab)$
for all $a,b\in \A$, and since $\phi$ is continuous, $T$ is
continuous.
\end{proof}
\begin{thm}\label{Jzp}
Let $\X$ be a linear space and let $\PH$ be a bilinear map
satisfying
\[ a,b\in \A, \quad a\circ b=0 \Rightarrow \phi(a,b)=0.\]
Then
\[ \phi(a,x)= \frac{1}{2}\phi(ax,1)+\frac{1}{2}\phi(xa,1) \]
for all $a\in \A$ and $x\in \Im(\A)$. Indeed, if $\A=\Im(\A)$,
then $\A$ is zero Jordan product determined.

\end{thm}
\begin{proof}
Let $a,p\in \A$ with $p^{2}=p$ and let $q=1-p$. We have $(p-q)
\circ paq=0$ and $(p-q)\circ qap=0$. So $\phi(paq,p-q)=0$ and
$\phi(qap,p-q)=0$. Hence $\phi(paq,p)=\phi(paq,q)$ and
$\phi(qap,p)=\phi(qap,q)$. Therefore
\begin{equation*}
\begin{split}
& \phi(paq,p)=\frac{1}{2}\phi(paq,1);\, and\\ &
\phi(qap,p)=\frac{1}{2}\phi(qap,1).
\end{split}
\end{equation*}
By these identities and the fact that $pap\circ q=0$ and
$qaq\circ p=0$, we have
\begin{equation*}
\begin{split}
&\frac{1}{2}\phi(ap,1)+\frac{1}{2}\phi(pa,1)=
\frac{1}{2}\phi(pap,p)+\frac{1}{2}\phi(qap,p)+\\&
\frac{1}{2}\phi(qap,q)+\frac{1}{2}\phi(pap,p)+
\frac{1}{2}\phi(paq,p)+\frac{1}{2}\phi(paq,q)=\\&
\phi(pap,p)+\frac{1}{2}\phi(qap,1)+\frac{1}{2}\phi(paq,1)=\\&
\phi(pap+qap+paq+qaq,p)=\phi(a,p).
\end{split}
\end{equation*}
Since every $x\in \Im(\A)$ is a linear combination of idempotent
elements in $\A$, we get
\[ \phi(a,x)= \frac{1}{2}\phi(ax,1)+\frac{1}{2}\phi(xa,1)\]
for all $a\in \A$ and $x\in \Im(\A)$.
\par
Now let $\A=\Im(\A)$, $\X$ be a linear space, and let $\PH$ be a
bilinear map such that for all $a,b \in \A$, $a\circ b= 0$
implies $\phi(a,b)=0$. If we define the linear map $T:\A
\rightarrow \X$ by $T(a)=\frac{1}{2}\phi(a,1)$, then $T$
satisfies all the requirements in the definition of zero Jordan
product determined algebras. Thus $\A$ is a zero Jordan product
determined algebra.
\end{proof}
\begin{prop}\label{BJzp}
Let $\A$ be a Banach algebra, let $\X$ be a Banach space and let
$\PH$ be a continuous bilinear map satisfying
\[ a,b\in \A, \quad a\circ b=0 \Rightarrow \phi(a,b)=0.\]
Then
\[ \phi(a,x)= \frac{1}{2}\phi(ax,1)+\frac{1}{2}\phi(xa,1)\]
for all $a\in \A$ and $x\in \overline{\Im(\A)}$. If
$\A=\overline{\Im(\A)}$, then there exists a continuous linear
map $T: \A \rightarrow \X$ such that $\phi(a,b) =T(a\circ b)$ for
all $a,b\in \A$.
\end{prop}
\begin{proof}
By using similar arguments as that in the proof of
Theorem~\ref{Jzp} and the fact that $\phi$ is continuous, it
follows that $\phi(a,x)=
\frac{1}{2}\phi(ax,1)+\frac{1}{2}\phi(xa,1)$ for all $a\in \A$
and $x\in \overline{\Im(\A)}$. If $\A=\overline{\Im(\A)}$, we get
\[ \phi(a,b)= \frac{1}{2}\phi(ab,1)+\frac{1}{2}\phi(ba,1)\]
for all $a,b\in \A$. Define $T:\A \rightarrow \X$ by
$T(a)=\frac{1}{2}\phi(a,1)$. Then $T$ is continuous and
$\phi(a,b) =T(a\circ b)$ for all $a,b\in \A$.
\end{proof}
We continue by studying the condition $(G)$.
\begin{thm}\label{ds}
Let $\X$ be a linear space and let $\PH$ be a bilinear map
satisfying $(G)$. Then
\[ \phi(a,x)+\phi(x,a)=\phi(ax,1)+\phi(1,xa) \quad and \quad \phi(x,1)=\phi(1,x)\]
for all $a\in \A$ and $x\in \Im(\A)$. Indeed, if $\A=\Im(\A)$,
then
\[ \phi(a,b)+\phi(b,a)=\phi(ab,1)+\phi(1,ba)\quad and \quad \phi(a,1)=\phi(1,a)\]
for all $a,b\in \A$.
\end{thm}
\begin{proof}
Let $a,p\in \A$ with $p^{2}=p$ and let $q=1-p$. Since $pq=qp=0$,
we see that
\[ \phi(p,q)=\phi(p,1)-\phi(p,p)=0 \quad and \quad \phi(q,p)=\phi(1,p)-\phi(p,p)=0.\]
So $\phi(p,1)=\phi(1,p)$. By linearity, it shows
\[\phi(x,1)=\phi(1,x)\]
for all $x\in \Im(\A)$. Now we have
$(p+paq)(q-paq)=(q-paq)(p+paq)=0$ and
$(p+qap)(q-qap)=(q-qap)(p+qap)=0$. So $\phi(p+paq,q-paq)=0$ and
$\phi(p+qap,q-qap)=0$. Hence
\[\phi(paq,p)=\phi(q,paq)\quad and \quad
\phi(p,qap)=\phi(qap,q).\]
By these identities and the fact that $(pap)q=q(pap)=0$ and
$(qaq)p=p(qaq)=0$, we have
\begin{equation*}
\begin{split}
\phi(a,p)+\phi(p,a)&=\phi(pap,p)+\phi(paq,p)+\phi(qap,p)\\&+\phi(p,pap)+\phi(p,paq)+\phi(p,qap)\\&
=\phi(pap,p)+\phi(q,paq)+\phi(qap,p)\\&+\phi(p,pap)+\phi(p,paq)+\phi(qap,q)\\&
=\phi(ap,1)+\phi(1,pa)
\end{split}
\end{equation*}
Since every $x\in \Im(\A)$ is a linear combination of idempotent
elements in $\A$, we get
\[ \phi(a,x)+\phi(x,a)=\phi(ax,1)+\phi(1,xa) \]
for all $a\in \A$ and $x\in \Im(\A)$.
\end{proof}
\begin{cor}\label{BA}
Let $\A$ be a Banach algebra, let $\X$ be a Banach space and let
$\PH$ be a continuous bilinear map satisfying $(G)$. Then
\[ \phi(a,x)+\phi(x,a)=\phi(ax,1)+\phi(1,xa)\quad and \quad \phi(x,1)=\phi(1,x) \]
for all $a\in \A$ and $x\in \overline{\Im(\A)}$. Indeed, if
$\A=\overline{\Im(\A)}$, then
\[ \phi(a,b)+\phi(b,a)=\phi(ab,1)+\phi(1,ba)\quad and \quad \phi(a,1)=\phi(1,a) \]
for all $a,b\in \A$.
\end{cor}
Recall that a bilinear map $\PH$, where $\X$ is a linear space,
is called \emph{symmetric} if $\phi(a,b)=\phi(b,a)$ holds for all
$a,b\in \A$.
\par
\begin{prop}\label{n}
Let $\A=\Im(\A)$ ($\A=\overline{\Im(\A)}$), let $\X$ be a linear
(Banach) space and let $\PH$ be a (continuous) bilinear map. The
following conditions are equivalent:
\begin{enumerate}
\item[(i)] $\phi$ is a symmetric bilinear map satisfying the condition
\[ a,b\in \A, \quad ab=ba=0 \Rightarrow \phi(a,b)=0;\]
\item[(ii)] $\phi$ satisfies
\[ a,b\in \A, \quad a\circ b=0 \Rightarrow \phi(a,b)=0;\]
\item[(iii)] there exists a (continuous) linear map $T: \A \rightarrow \X$ such that $\phi(a,b)=T(a\circ
b)$ for all $a,b\in \A$.
\end{enumerate}
\end{prop}
\begin{proof}
$(iii)\Rightarrow (i)$ and $(iii)\Rightarrow (ii)$ are clear.
$(ii)\Rightarrow (iii)$ obtains from Theorem~\ref{Jzp}
(Proposition~\ref{BJzp}). We show that $(i)\Rightarrow (ii)$
holds.
\par
By Theorem~\ref{ds} (Corollary~\ref{BA}), we have
\[ \phi(a,b)+\phi(b,a)=\phi(ab,1)+\phi(1,ba) \]
for all $a,b\in \A$. So $\phi(a,b)=\frac{1}{2}\phi(ab+ba,1)$,
since $\phi$ is symmetric. If we define the linear mapping $T:\A
\rightarrow \X$ by $T(a)=\frac{1}{2}\phi(a,1)$, then
$\phi(a,b)=T(a\circ b)$ for all $a,b\in \A$ (It is obvious if
$\phi$ is continuous, then $T$ is continuous).
\end{proof}
\section{Characterizing derivations and Jordan derivations through zero products}
In this section for $\M$ bimodule over $\A$, and $D:\A \rightarrow
\M$ a linear map, we will consider the following conditions:
\begin{enumerate}
\item[(d1)] $ab=0 \Rightarrow aD(b)+D(a)b=0.$
\item[(d2)] $ab=ba=0 \Rightarrow aD(b)+D(a)b=0.$
\item[(d3)] $a\circ b=0 \Rightarrow a\bullet D(b)+D(a)\bullet b=0.$
\item[(d4)] $ab=ba=0 \Rightarrow a\bullet D(b)+D(a)\bullet b=0.$
\end{enumerate}
\begin{thm}\label{d1}
Let $\A$ be an (Banach) algebra, $\M$ be an (Banach) $\A$-bimodule
and $\mathcal{J}$ be an ideal of $\A$ such that
$\mathcal{J}\subseteq \Im(\A)$ ($\mathcal{J}\subseteq
\overline{\Im(\A)}$) and
\[ \{ m\in \M \, | \, xm=mx=0  \, \, for
\, all \,\, x\in \mathcal{J}\}=\{0\} . \] Assume that $D:\A
\rightarrow \M$ is a (continuous) linear map satisfying
\emph{(d1)}. Then $D$ is a generalized derivation and
$aD(1)=D(1)a$ for all $a\in \A$.
\end{thm}
\begin{proof}
Define a bilinear map $\phi:\A \times \A \rightarrow \M$ by
$\phi(a,b)=aD(b)+D(a)b$. Then $\phi(a,b)=0$ for all $a,b\in \A$
with $ab=0$. By applying Theorem~\ref{zp}, we obtain
$\phi(a,x)=\phi(ax,1)$ for all $a\in \A$ and $x\in \Im(\A)$. So
\begin{equation}\label{e1}
 aD(x)+D(a)x=axD(1)+D(ax),
\end{equation}
for all $a\in \A$ and $x\in \Im(\A)$. Letting $a=1$ in
\eqref{e1}, we arrive at $D(1)x=xD(1)$, for all $x\in
\mathcal{J}$. So we have $aD(1)x=axD(1)=D(1)ax$ and
$xaD(1)=D(1)xa=xD(1)a$, for all $a\in \A$ and $x\in \mathcal{J}$.
Hence $(aD(1)-D(1)a)\mathcal{J}=\mathcal{J}(aD(1)-D(1)a)=\{0\}$,
for each $a\in \A$. From hypothesis it follows that
\begin{equation}\label{e2}
 D(1)a=aD(1),
\end{equation}
for all $a \in \A$.
\par
Let $a,b \in \A$ and $x\in \Im(\A)$. By applying \eqref{e1} and
\eqref{e2}, we obtain
\begin{equation*}
D(abx)=abD(x)+D(ab)x-aD(1)bx,
\end{equation*}
and on the other hand
\begin{equation*}
\begin{split}
D(abx)& =aD(bx)+D(a)bx-abxD(1)\\ &= abD(x)+aD(b)x+D(a)bx-2aD(1)bx.
\end{split}
\end{equation*}
By comparing the two expressions for $D(abx)$, we arrive at
\begin{equation}\label{e3}
(D(ab)-aD(b)-D(a)b+aD(1)b)x=0
\end{equation}
for all $a,b\in \A$ and $x\in \Im(\A)$. By Theorem~\ref{zp}, we
have $\phi(x,a)=\phi(1,xa)$ for all $a\in \A$ and $x\in \Im(\A)$.
Now by this identity and using similar arguments as above it
follows that
\begin{equation}\label{e4}
x(D(ab)-aD(b)-D(a)b+aD(1)b)=0
\end{equation}
for all $a,b \in \A$ and $x\in \Im(\A)$. Hence from \eqref{e3} and
\eqref{e4}, we find that
$(D(ab)-aD(b)-D(a)b+aD(1)b)\mathcal{J}=\mathcal{J}(D(ab)-aD(b)-D(a)b+aD(1)b)=\{0\}$,
for each $a,b \in \A$. From hypothesis it follows that
\[ D(ab)=aD(b)+D(a)b-aD(1)b,\]
for all $a,b\in \A$.
\par
By Proposition~\ref{bzp} and using similar arguments as that in
the above proof , we get the result in case of Banach algebras.
\end{proof}
In order to prove next theorem we will adopt the following
notational convention
\[ [a,m,b]=amb+bma \quad and \quad [a,b,m]=[m,b,a]=abm+mba\]
for all $a,b\in \A$ and $m\in \M$, where $\A$ is an algebra and
$\M$ is an $\A$-bimodule. Also we need the following lemma, the
proof of which is routine and will be omitted.
\begin{lem}\label{f}
Let $\A$ be an algebra and $\M$ be an $\A$-bimodule. For all
$a,b,c \in \A$ and $m\in \M$ we have
\begin{enumerate}
\item[(i)]
\[ 2[a,m,b]=a\bullet (b\bullet m)+b\bullet (a\bullet m)-(a\circ b)\bullet m
\]
and
\[ 2[a,b,m]=a\bullet (b\bullet m)+(a\circ b)\bullet m-b\bullet (a\bullet
m);
\]

\item[(ii)]
\[ [m,a\circ b,c]=[b\bullet m,a,c]+[m,a,b\circ c]-[m,a,c]\bullet b\]
and
\begin{equation*}
\begin{split}
[a, b\bullet m, c]&=[a\bullet m,b,c]+[a,b,c\bullet
m]-[a,b,c]\bullet m
\\&=[a\circ b,m,c]+[a,m,b\circ c]-[a,m,c]\bullet b.
\end{split}
\end{equation*}
\end{enumerate}
\end{lem}
\begin{thm}\label{d2}
Let $\A$ be an (Banach) algebra, $\M$ be an (Banach) $\A$-bimodule
satisfying $\mathbb{M}$ ($\mathbb{M}^{\prime}$). Suppose that
$D:\A \rightarrow \M$ is a (continuous) linear map. Then the
following conditions are equivalent:
\begin{enumerate}
\item[(i)] $D$ is a generalized Jordan
derivation and $aD(1)=D(1)a$ for all $a\in \A$;
\item[(ii)] $D$ satisfies \emph{(d3)};
\item[(iii)] $D$ satisfies \emph{(d4)}.
\end{enumerate}
\end{thm}
\begin{proof}
Clearly (i) implies (ii) and (ii) implies (iii). We show that
(iii) implies (i).
\par
Let $\mathcal{J}$ be an ideal of $\A$ such that
$\mathcal{J}\subseteq \Im(\A)$ (if $\A$ is a Banach algebra we
assume that $\mathcal{J}\subseteq \overline{\Im(\A)}$) and
\[\{m\in \M \, | \, xmx=0 \, \, for \, all \,\, x\in
\mathcal{J}\}=\{0\}.\]
Let $p$ be a idempotent of $\A$. As
$p(1-p)=(1-p)p=0$ it follows that
\[ 2D(p)+pD(1)+D(1)p=2pD(p)+2D(p)p.\]
By multiplying this identity on the left and right by $p$,
respectively, we arrive at
\begin{equation*}
\begin{split}
&  pD(1)p+D(1)p=2pD(p)p,\\& pD(1)+pD(1)p=2pD(p)p,
\end{split}
\end{equation*}
which implies $pD(1)=D(1)p$. By linearity, it shows $xD(1)=D(1)x$
for all $x\in \mathcal{J}$. Hence $aD(1)x=D(1)ax$ and
$xD(1)a=xaD(1)$ for each $a \in \A$ and $x\in \mathcal{J}$.
Therefore $x(aD(1)-D(1)a)x=0$ for all $x\in \mathcal{J}$ and by
hypothesis we have
\[ aD(1)=D(1)a \]
for all $a \in \A$.
\par
Define $\Delta:\A\rightarrow \M$ by $\Delta(a)=D(a)-aD(1)$. Then
$\Delta$ is a linear map which satisfies (d4) and $\Delta(1)=0$.
We will show that $\Delta$ is a Jordan derivation. So $D$ is a
generalized Jordan derivation.
\par
Now define a bilinear map $\phi:\A \times \A \rightarrow \M$ by
$\phi(a,b)=a\bullet \Delta(b)+\Delta(a) \bullet b$. So
$\phi(a,b)=0$ for all $a,b\in \A$ with $ab=ba=0$, and by
Theorem~\ref{ds}, we get
$\phi(a,x)+\phi(x,a)=\phi(ax,1)+\phi(1,xa)$ for all $a\in \A$ and
$x\in \mathcal{J}$. Hence
\begin{equation}\label{e5}
\Delta(a\circ x)=a\bullet \Delta(x)+\Delta(a)\bullet x
\end{equation}
for all $a \in \A$ and $x\in \mathcal{J}$.
\\ \\
\textbf{Claim1.} For all $a \in \A$ and $x,y \in \mathcal{J}$, we
have
\[\Delta([x,a,y])=[\Delta(x),a,y]+[x,\Delta(a),y]+[x,a,\Delta(y)]
\]
\\
\textbf{Reason.} Let $x,y \in \mathcal{J}$ and $a\in \A$. From
Lemma~\ref{f} and \eqref{e5}, we obtain
\begin{equation*}
\begin{split}
2\Delta([x,a,y])&=\Delta(x\circ(a\circ y))+\Delta(y\circ (a\circ
x))-\Delta((x\circ y)\circ a)\\&= x\bullet \Delta(a\circ
y)+\Delta(x)\bullet(a\circ y)+y\bullet \Delta(a\circ
x)\\&+\Delta(y)\bullet (a\circ x)-(x\circ y)\bullet
\Delta(a)-\Delta(x\circ y)\bullet a\\&= x\bullet (y \bullet
\Delta(a))+x\bullet (\Delta(y)\bullet a)+\Delta(x)\bullet (y\circ
a)\\&+y\bullet(\Delta(x)\bullet a)+y\bullet (\Delta(a)\bullet
x)+\Delta(y)\bullet (x \circ a)\\&- (x \circ y)\bullet \Delta
(a)-(x\bullet \Delta(y))\bullet a-(\Delta(x)\bullet y)\bullet
a\\&= 2[\Delta(x),a,y]+2[x,\Delta(a),y]+2[x,a,\Delta(y)].
\end{split}
\end{equation*}
\\
\textbf{Claim2.} For all $a \in \A$ and $x,y \in \mathcal{J}$, we
have
\[\Delta([x,a^{2},y])=[\Delta(x),a^{2},y]+[x,a\bullet \Delta(a),y]+[x,a^{2},\Delta(y)]
\]
\\
\textbf{Reason.} Let $x,y \in \mathcal{J}$ and $a\in \A$. From
this Lemma~\ref{f}, Claim 1 and \eqref{e5}, it follows that
\begin{equation*}
\begin{split}
2\Delta([x,a^{2},y])&=\Delta([x,a\circ a,y])\\&= \Delta([x\circ
a,a,y])+\Delta ([x,a,y\circ a])-\Delta([x,a,y]\circ
a)\\&=[\Delta(x\circ a),a,y]+[x\circ a,\Delta(a),y]+[x\circ
a,a,\Delta(y)]\\&+ [\Delta(x),a,y\circ a]+[x,\Delta(a),y\circ
a]+[x,a,\Delta(y\circ a)]\\&-a\bullet
\Delta([x,a,y])-\Delta(a)\bullet [x,a,y].
\end{split}
\end{equation*}
So
\begin{equation*}
\begin{split}
2\Delta([x,a^{2},y])&= [a\bullet
\Delta(x),a,y]+[\Delta(x),a,y\circ a]-[\Delta(x),a,y]\bullet
a\\&+[\Delta(a)\bullet x,a,y]+[x,a,y\bullet
\Delta(a)]-[x,a,y]\bullet \Delta(a)\\&+[x,a,a\bullet
\Delta(y)]+[a\circ x,a,\Delta(y)]-[x,a,\Delta(y)]\bullet
a\\&+[x\circ a,\Delta(a),y]+[x,\Delta(a),y\circ
a]-[x,\Delta(a),y]\bullet a\\&=2[\Delta(x),a^{2},y]+2[x,a\bullet
\Delta(a),y]+2[x,a^{2},\Delta(y)].
\end{split}
\end{equation*}
\par
Now by applying Claim 1, we have
\[\Delta([x,a^{2},x])=[\Delta(x),a^{2},x]+[x,\Delta(a^{2}),x]+[x,a^{2},\Delta(x)]
\]
for all $a \in \A$ and $x \in \mathcal{J}$. On the other hand
from Claim 2, we see that
\[\Delta([x,a^{2},x])=[\Delta(x),a^{2},x]+[x,a\bullet \Delta(a),x]+[x,a^{2},\Delta(x)]
\]
for all $a \in \A$ and $x \in \mathcal{J}$. By comparing the two
expressions for $\Delta([x,a^{2},x])$, we arrive at
\[ x(\Delta(a^{2})-a\bullet \Delta(a))x=0\]
for all $a \in \A$ and $x \in \mathcal{J}$. Therefore by
hypothesis we have $\Delta(a^{2})=a\bullet \Delta(a)$ for each
$a\in \A$ and so $\Delta$ is a Jordan derivation.
\par
Similarly, by Corollary~\ref{BA} we have the result in case of
Banach algebras and continuous linear maps.
\end{proof}\label{Bd2}
\begin{thm}\label{dd2}
Let $\A$ be an (Banach) algebra, $\M$ be an (Banach) $\A$-bimodule
satisfying $\mathbb{M}$ ($\mathbb{M}^{\prime}$). Suppose that
$D:\A \rightarrow \M$ is a (continuous) linear map satisfying
\emph{(d2)}. Then $D$ is a generalized Jordan derivation and
$aD(1)=D(1)a$ for all $a\in \A$.
\end{thm}
\begin{proof}
Let $a,b\in \A$ with $ab=ba=0$. So
\[ aD(b)+D(a)b=0 \quad and \quad bD(a)+D(b)a=0.\]
Hence $aD(b)+D(a)b+bD(a)+D(b)a=0$ and $D$ satisfies (d4).
Therefore by Theorem~\ref{d2}, $D$ is a generalized Jordan
derivation and $aD(1)=D(1)a$ for all $a\in \A$.
\end{proof}
\begin{rem}
In Theorem~\ref{dd2} it is not necessarily true that any linear
mapping $D:\A\rightarrow \M$ satisfying (d2) is a generalized
derivation. Indeed, if $D$ is a \emph{anti-derivation}, i.e.
$D(ab)=D(b)a+bD(a)$ for all $a,b\in \A$, then $D$ satisfies (d2).
There are simple examples on some algebras and their (special)
bimodules with anti-derivations such that they are not
derivations. An example is given on the algebra $T_{2}$ of $2
\times 2$ upper triangular matrices over $\mathbb{C}$ \cite{Jon}.
Let us recall it. We make $\mathbb{C}$ an $T_{2}$-bimodule by
defining $a\gamma =a_{22}\gamma$ and $\gamma a=\gamma a_{11}$ for
all $\gamma\in \mathbb{C}$, $a\in T_{2}$. A map $D:T_{2}
\rightarrow \mathbb{C}$ defined by $D(a)=a_{12}$ is an
anti-derivation which is not a derivation. Note that if
$\A=T_{2}$ and $\M=\mathbb{C}$, then $\A$, $\M$ and $D$ satisfy
all the requirements in Theorem~\ref{dd2}.
\end{rem}

\subsection*{Acknowledgment}
The author like to express his sincere thanks to the referees for
this paper.

\bibliographystyle{amsplain}
\bibliography{xbib}

\end{document}